\documentclass[11pt]{article}

\usepackage[margin=20mm]{geometry}
\usepackage{amsmath,amssymb,amsthm,mathtools}
\usepackage{stmaryrd}
\usepackage{xcolor}
\usepackage{hyperref}
\usepackage{todonotes}

\theoremstyle{plain}
\newtheorem{theorem}{Theorem}[section]
\newtheorem*{theorem*}{Theorem}
\newtheorem{definition}[theorem]{Definition}
\newtheorem*{definition*}{Definition}
\newtheorem{proposition}[theorem]{Proposition}
\newtheorem*{proposition*}{Proposition}
\newtheorem{lemma}[theorem]{Lemma}
\newtheorem*{lemma*}{Lemma}
\newtheorem{corollary}[theorem]{Corollary}
\newtheorem*{corollary*}{Corollary}

\newtheorem{example}[theorem]{Example}
\newtheorem*{example*}{Example}
\newtheorem{remark}[theorem]{Remark}
\newtheorem*{remark*}{Remark}

\newtheorem*{notation*}{Notation}

\hypersetup{
  colorlinks=true,
  linkcolor=blue,
  filecolor=magenta,
  urlcolor=cyan,
  citecolor=red,
  pdftitle={An analytic approach to the finite R-transform},
  pdfauthor={Octavio Arizmendi and Katsunori Fujie},
}

\newcommand{\calP}{\mathcal{P}}

\newcommand{\calM}{\mathcal{M}}

\DeclareMathOperator{\supp}{supp}

\newcommand{\meas}[1]{\mu\llbracket #1 \rrbracket}
\newcommand{\Lap}[1]{\mathcal{L}[#1]}
\newcommand{\abs}[1]{\left| #1 \right|}

\newcommand{\weakto}{\xrightarrow{w}}

\newcommand{\coef}{\widetilde{e}}

\newcommand{\N}{\mathbb{N}}

\newcommand{\R}{\mathbb{R}}
\newcommand{\C}{\mathbb{C}}
\newcommand{\E}{\mathbb{E}}

\renewcommand{\epsilon}{\varepsilon}

\numberwithin{equation}{section}

\begin{document}

\title{An analytic approach to the finite $R$-transform}
\author{Octavio Arizmendi and Katsunori Fujie}
\date{\today}

\maketitle

\begin{abstract}
We revisit Marcus' finite free analogue of Voiculescu $R$-transform from an analytic viewpoint.
By relating the finite free Fourier transform to the Laplace transform, we study the finite $R$-transform through logarithmic potentials and Legendre transforms.
Under suitable assumptions, we prove that the finite $R$-transform of a polynomial differs from the Voiculescu $R$-transform of its empirical root distribution by $O(N^{-1})$.
As an application, we obtain an analytic proof of the convergence of finite free additive convolution to free additive convolution.
\end{abstract}

\section{Introduction}

Finite free probability is a relatively recent branch of probability theory, initiated by the work of Marcus, Spielman, and Srivastava on the resolution of the Kadison--Singer problem and the construction of Ramanujan graphs \cite{MSS1,MSS2,MSS22}.
Its central theme is to develop polynomial analogues of the results in free probability.
By now, many such analogies have been established; see \cite{AFPU, AGVP, AP, F, JSS, MSS22} for example.

Let $\calP_N(\C)$ denote the set of monic polynomials over $\C$ of degree $N$.
For $p \in \calP_N(\C)$, we write
\[
  p(x)
  = \prod_{i=1}^N (x-\lambda_i(p))
  = \sum_{k=0}^N x^{N-k} (-1)^k \binom{N}{k} \coef_k^{(N)}(p),
\]
where $\{\lambda_i(p)\}_{i=1}^N$ are the roots of $p$, and $\coef_k^{(N)}(p)$ denotes the $k$-th normalized elementary symmetric polynomial of the roots, namely
\[
  \coef_k^{(N)}(p)
  = \binom{N}{k}^{-1}
    \sum_{i_1<\cdots<i_k}
    \lambda_{i_1}(p)\cdots\lambda_{i_k}(p),
\]
with the convention $\coef_0^{(N)}(p)=1$.
Although this normalization may look unusual at first glance, it is particularly convenient in finite free probability.
Throughout the paper, we always identify a monic polynomial $p$, its roots $\{\lambda_i(p)\}_{i=1}^N$, and its coefficients $\{\coef_k^{(N)}(p)\}_{k=0}^N$.
We also associate to $p$ its empirical root distribution $\meas{p} = \frac{1}{N}\sum_{i=1}^N \delta_{\lambda_i(p)}.$
When the degree $N$ is clear from the context, we often omit the superscript $(N)$ and simply write $\coef_k(p)$.
For a subset $S\subset\C$, we write $\calP_N(S)$ for the subset of $\calP_N(\C)$ consisting of polynomials whose roots lie in $S$.
In particular, if $p\in\calP_N(\R)$, then $\meas{p}$ is a probability measure on $\R$.

For $p,q\in\calP_N(\C)$, the finite free additive convolution $p\boxplus_N q \in \calP_N(\C)$ is defined by
\[
  \coef_k(p\boxplus_N q)
  = \sum_{i=0}^k \binom{k}{i}\coef_i(p)\coef_{k-i}(q),
  \qquad k=0,\dots,N.
\]
This binary operation already appeared in classical work of Szeg\H{o} and Walsh \cite{Sz22,Walsh}; see also \cite{Mar66}.
A fundamental contribution of Marcus, Spielman, and Srivastava was the discovery of its random matrix expression:
if
\[
  p(x)=\det(xI-A),
  \qquad
  q(x)=\det(xI-B),
\]
where $A$ and $B$ are $N\times N$ normal matrices, then
\[
  p\boxplus_N q(x)
  = \E_U\!\left[\det(xI-A-UBU^*)\right],
\]
where $U$ is an $N\times N$ Haar unitary matrix.
This formula reveals a close connection between finite free probability and free probability, since the random matrices $A$ and $UBU^*$ are asymptotically free in the large $N$ limit.

Motivated by this connection, one is naturally led to ask whether finite free convolution approximates free additive convolution in the large-degree limit.
More precisely, if $p_N,q_N\in\calP_N(\R)$ satisfy $\meas{p_N}\weakto\mu$ and $\meas{q_N}\weakto\nu$, it is natural to expect that $\meas{p_N\boxplus_N q_N}\weakto \mu\boxplus\nu$, where $\mu\boxplus\nu$ denotes the free additive convolution of probability measures $\mu,\nu\in\calM(\R)$.
This convergence is by now well established, see \cite{AP, F, Marcus}.
We provide an alternative proof from a direct analytic study of the finite $R$-transform, see Corollary \ref{cor:finitefreeadditive}.

A basic tool for studying $\mu\boxplus\nu$ is Voiculescu $R$-transform, which linearizes free additive convolution.
It plays the role in free probability analogous to that of the logarithm of the characteristic function (Fourier transform) in classical probability.
Recall that, for a probability measure $\mu\in\calM(\R)$, the Cauchy transform is defined by
\[
  G_\mu(z)=\int_\R \frac{1}{z-t}\,d\mu(t),
  \qquad
  z\in\C\setminus\supp(\mu).
\]
It is known that $G_\mu$ is univalent in a suitable region near infinity, and the $R$-transform is then defined by
\[
  R_\mu(z)=G_\mu^{-1}(z)-\frac{1}{z}
\]
for $z$ near $0$.
Its key property is the linearization formula: $R_{\mu\boxplus\nu}=R_\mu+R_\nu$.
The $R$-transform and the associated subordination theory provide a powerful framework for analyzing free convolution; see \cite{BV92, BV98, MS, NS, V85, V86}.

The convergence of finite free convolution to free convolution was first suggested by Marcus through a finite analogue of the $R$-transform \cite{Marcus} .
Roughly speaking, Marcus observed that this finite transform should converge to Voiculescu $R$-transform as the degree tends to infinity.
However, the analytic content of this observation was not fully developed there.
Later, Arizmendi and Perales \cite{AP} gave a rigorous proof of the convergence of finite free convolution by introducing finite free cumulants, namely the coefficients of Marcus' finite $R$-transform, and showing that these cumulants converge to the free cumulants of the limiting measure.
Their argument is purely combinatorial and works under compact support assumptions.
See also \cite[Section 3]{MergnyPotters2022} for related remarks on the convergence of the finite $R$-transform.

The purpose of the present paper is to revisit Marcus' finite $R$-transform from an analytic point of view and to make its asymptotic relation with Voiculescu $R$-transform precise.
Our starting point is the finite free Fourier (FFF) transform; for $p(x)=\sum_{k=0}^N x^{N-k}(-1)^k\binom{N}{k}\coef_k^{(N)}(p)\in\calP_N(\C)$, define
\[
  \widehat{p}(s)
  = \sum_{k=0}^N \frac{(-1)^k\coef_k^{(N)}(p)}{k!} s^k
  = \sum_{k=0}^N \frac{\coef_k^{(N)}(p)}{k!}(-s)^k.
\]
For formal power series $f,g\in\C[[s]]$, we write $f \overset{N}{=} g$ if $f\equiv g \pmod{s^{N+1}}$.
Then one can check that $\widehat{p}(\partial_x)x^N = p(x),$ and $r=p\boxplus_N q$ if and only if $\widehat{r}(s)\overset{N}{=}\widehat{p}(s)\widehat{q}(s)$ for $p,q,r\in\calP_N(\C)$.
This motivates the definition of the finite $R$-transform by Marcus:
\[
  R_p^{(N)}(s)
  = -\frac{1}{N} \frac{d}{ds} \log \widehat{p}(Ns),
\]
which will be discussed more precisely in Section~2.
Thus, the finite $R$-transform linearizes finite free convolution $R_{p\boxplus_N q}^{(N)} \overset{N-1}{=} R_p^{(N)} + R_q^{(N)}$ in the sense of equality modulo $s^N$.

A key idea is that the FFF transform $\widehat{p}$ can be expressed in terms of the Laplace transform of $p$.
This allows us to study $R_p^{(N)}$ through logarithmic potentials, Legendre transforms, and a quantitative Laplace principle.
In this way, the finite $R$-transform acquires an analytic interpretation and can be compared directly with the Voiculescu $R$-transform of the empirical root distribution.

Our main result is as follows.
\begin{theorem} \label{thm:main}
Let $\epsilon > 0$.
Assume that $p_N \in \calP_N((-\infty, -\epsilon])$ is a sequence of polynomials whose empirical root distributions $\meas{p_N}$ converge weakly to a probability measure $\mu \in \calM((-\infty, -\epsilon])$, and set $\alpha:= G_\mu(0)$.
Then, for each $s \in (0, \alpha)$, we have
\[
  R_{p_N}^{(N)}(s) - R_{\meas{p_N}}(s) = O(N^{-1}).
\]
\end{theorem}

As an application, we recover the convergence of finite free additive convolution to free additive convolution; see Corollary~\ref{cor:finitefreeadditive}.

The paper is organized as follows.
In Section~2, we introduce the necessary notation and present the analytic framework underlying the finite $R$-transform.
In Section~3, we prove the quantitative convergence theorem and discuss several applications.

\section{Preliminaries}

In this section, we review the necessary background and develop the analytic framework underlying the finite $R$-transform.

\subsection{\texorpdfstring{Voiculescu $R$-transform}{Voiculescu R-transform}}

For a probability measure $\mu \in \calM(\R)$, the Cauchy transform is defined by $G_\mu(z) = \int_{\R} \frac{1}{z-t}\,d\mu(t),$ for $z \in \C \setminus \supp(\mu)$,
and the $R$-transform is given by $R_\mu(z) = G_\mu^{-1}(z) - \frac{1}{z}$ for $z$ near $0$.
A basic difficulty in working with the $R$-transform is that it is essentially defined as the inverse of the Cauchy transform.
For our purposes, it is more convenient to reformulate this in terms of logarithmic potentials and Legendre duality.

Throughout this paper, unless otherwise stated, we assume that $\mu$ is supported on the non-positive real line.

\begin{definition}[Logarithmic potential]
  Let $\mu \in \calM((-\infty,0])$.
  Assume that $\int_{-\infty}^0 \abs{\log(z-t)}\,d\mu(t) < \infty$ for all $z>0$.
  The logarithmic potential of $\mu$ is defined by
  \[
    H_\mu(z) = \int_{-\infty}^0 \log(z-t)\,d\mu(t),
    \qquad z>0.
  \]
\end{definition}

We next recall a simple variant of the Legendre transform adapted to concave functions.
Although the classical Legendre transform is usually formulated for convex functions,
the following concave version is sufficient for our purpose.

\begin{definition}[Concave Legendre transform]
Let $h$ be a differentiable strictly concave function on $(0,\infty)$.
Its (concave) Legendre transform is defined by
\[
  h^*(s) := \inf_{x>0} \{ sx - h(x) \}.
\]
\end{definition}

The following property is the analogue of the usual Legendre duality.

\begin{proposition}
Let $h$ be differentiable and strictly concave on $(0,\infty)$.
Then
\[
  h^*(h'(x)) = xh'(x) - h(x),
\]
and
\[
  (h^*)'(h'(x)) = x.
\]
In other words, the derivative of $h^*$ is the inverse function of $h'$.\end{proposition}

We now apply this to the logarithmic potential.
Since
\[
  \frac{d}{dz} H_\mu(z) = G_\mu(z)
\]
and $G_\mu$ is strictly decreasing on $(0,\infty)$, the function $H_\mu$ is strictly concave on $(0,\infty)$.
Hence we may consider its Legendre transform
\[
  H_\mu^*(s) = \inf_{x>0} \{ sx - H_\mu(x) \}.
\]

Let $\alpha := \lim_{x\downarrow 0} G_\mu(x) \in (0,\infty].$
Since $G_\mu(+\infty)=0$, the function $H_\mu^*(s)$ is naturally defined for $s\in(0,\alpha)$.

\begin{corollary}
For $s\in(0,\alpha)$, we have
\[
  \frac{d}{ds} H_\mu^*(s)
  = G_\mu^{(-1)}(s)
  = R_\mu(s) + \frac{1}{s}.
\]
\end{corollary}

\subsection{\texorpdfstring{Finite $R$-transform and Laplace transform}{Finite R-transform and Laplace transform}}

For $p(x)=\sum_{k=0}^N x^{N-k}(-1)^k\binom{N}{k}\coef_k^{(N)}(p)\in\calP_N(\C)$, the finite free Fourier (FFF) transform is
\[
  \widehat{p}(s)
  = \sum_{k=0}^N \frac{(-1)^k\coef_k^{(N)}(p)}{k!}s^k.
\]
Then, $\widehat{p\boxplus_N q}(s)\overset{N}{=}\widehat{p}(s)\widehat{q}(s)$ for $p,q\in\calP_N(\C)$.
This motivates the logarithmic FFF transform
\[
  C_p^{(N)}(s) := \log \widehat{p}(s),
\]
which is well defined near the origin since $\widehat{p}(0)=1$.
Then
\begin{equation}\label{eq:C-linearization}
  C_{p\boxplus_N q}^{(N)}(s)\overset{N}{=}C_p^{(N)}(s)+C_q^{(N)}(s).
\end{equation}

Expanding
\[
  C_p^{(N)}(s)=\sum_{n=1}^\infty \frac{(-1)^n c_n^{(N)}(p)}{n!}s^n,
\]
we have the following moment-cumulant formula
\[
    \coef_n(p) = \sum_{\pi \in P(n)} c_\pi^{(N)}(p), \qquad n \in \N
\]
and equivalently
\begin{equation} \label{eq:c_nExpansion}
  c_n^{(N)}(p) = \sum_{\pi \in P(n)} \coef_\pi^{(N)}(p) \mu_n(\pi, 1_n), \qquad n \in \N
\end{equation}
where $P(n)$ denotes the set partitions of $[n] = \{1, \dots, n\}$ and $\mu_n$ the M\"obius function of $P(n)$; see \cite{AP, NS}.
In particular, Eq. \eqref{eq:C-linearization} implies
\begin{equation} \label{eq:c_nLinear}
  c_n^{(N)}(p \boxplus_N q) = c_n^{(N)}(p) + c_n^{(N)}(q)
\end{equation}
for $n=1, \dots, N$.

However, the function $C_p^{(N)}(s)$ does not converge to $R_\mu$ even when $\meas{p_N}\weakto\mu$; the reason for this will become clear later.
We therefore introduce the following finite $R$-transform instead.

\begin{definition}[Finite $R$-transform and finite free cumulants]
  Let $p\in\calP_N(\C)$.
  The finite $R$-transform of $p$ is defined by
  \[
    R_p^{(N)}(s)
    := -\frac{1}{N}\frac{d}{ds}C_p^{(N)}(Ns)
    = -\frac{1}{N}\frac{d}{ds}\log \widehat{p}(Ns) = - \frac{\widehat{p}'(Ns)}{\widehat{p}(Ns)}
  \]
We also define the finite free cumulants of $p$ as the coefficients of the expansion
\[
  R_p^{(N)}(s)
  = \sum_{n=0}^{\infty} \kappa_{n+1}^{(N)}(p)s^n.
\]
\end{definition}

From the definition and Eq. \eqref{eq:c_nExpansion}, the finite free cumulants satisfy
\[
\kappa_n^{(N)}(p)=\frac{(-N)^{n-1}}{(n-1)!}c_n^{(N)}(p) = \frac{(-N)^{n-1}}{(n-1)!}\sum_{\pi \in P(n)} \coef_\pi^{(N)}(p) \mu_n(\pi, 1_n), \qquad n \in \N.
\]

By Eq. \eqref{eq:C-linearization}, the finite $R$-transform linearizes finite free convolution, at least up to order $N-1$:
\begin{equation} \label{eq:linearityOfR}
  R_{p\boxplus_N q}^{(N)}(s)\overset{N-1}{=}R_p^{(N)}(s)+R_q^{(N)}(s);
\end{equation}
in other words, one has the linearization property of finite free cumulants
\begin{equation*}
  \kappa_n^{(N)}(p \boxplus_N q) = \kappa_n^{(N)}(p) + \kappa_n^{(N)}(q)
\end{equation*}
for $n=1, \dots, N$ by Eq. \eqref{eq:c_nLinear}.

We do not pursue the combinatorial theory of finite free cumulants any further and refer the interested reader to \cite{AP} for details.
Instead, we turn to the main analytic question: whether the finite $R$-transform converges to the usual $R$-transform in the large-degree limit. 
This problem forms the central theme of the present paper.
Our approach is to give an analytic interpretation of $R_p^{(N)}$. 
More precisely, we show that the FFF transform is closely related to the Laplace transform of the polynomial $p$, which leads to a natural integral representation of $R_p^{(N)}$. 
This representation will serve as the starting point of our analysis.

From now on, we assume that
\[
  p(x)=\prod_{i=1}^N (x+\lambda_i)\in\calP_N(\R_{<0}),
  \qquad \lambda_i>0.
\]

\begin{definition}[Laplace transform]
  Let $f:[0,\infty)\to\R$ be a measurable function.
  For $s>0$, the Laplace transform of $f$ is defined by
  \[
    \Lap{f}(s)=\int_0^\infty f(x)e^{-sx}\,dx,
  \]
  whenever the integral converges.
\end{definition}

The key observation is that the FFF transform coincides with the Laplace transform of $p$, up to a multiplicative factor.

\begin{proposition}\label{prop:Laplace-FFF}
  Let
  \[
    p(x)=\sum_{k=0}^N x^{N-k}(-1)^k\binom{N}{k}\coef_k(p).
  \]
  Then
  \[
    \widehat{p}(s)=\frac{s^{N+1}}{N!}\Lap{p}(s).
  \]
  Consequently,
  \[
    \log \widehat{p}(s)=\log \Lap{p}(s)+(N+1)\log s-\log N!,
  \]
  and therefore
  \[
    R_p^{(N)}(s)
    =-\frac{1}{N}\frac{d}{ds}\log \Lap{p}(Ns)-\frac{N+1}{Ns}.
  \]
\end{proposition}

\begin{proof}
  Since
  \[
    \Lap{x^k}(s)=\frac{k!}{s^{k+1}},
  \]
  we have
  \[
    \Lap{p}(s)
    =\sum_{k=0}^N (-1)^k\binom{N}{k}\coef_k(p) \frac{(N-k)!}{s^{N-k+1}}.
  \]
  Multiplying both sides by $s^{N+1}/N!$ yields
  \[
    \frac{s^{N+1}}{N!}\Lap{p}(s)
    =\sum_{k=0}^N \frac{(-1)^k\coef_k(p)}{k!}s^k
    =\widehat{p}(s),
  \]
  as claimed. The rest follows readily.
\end{proof}

\subsection{Integral representation and the Laplace principle}

Let $H_p$ and $G_p$ denote the logarithmic potential and the Cauchy transform of $\meas{p}$:
\[
  H_p(x)=\frac{1}{N}\log p(x)
  =\frac{1}{N}\sum_{i=1}^N \log(x+\lambda_i),
\]
and
\[
  G_p(x)=H_p'(x)
  =\frac{1}{N}\sum_{i=1}^N \frac{1}{x+\lambda_i}
  =\frac{p'(x)}{Np(x)}.
\]

Let
\[
  \alpha:=G_p(0)\in(0,\infty).
\]
For $s\in(0,\alpha)$, define
\[
  H_p^*(s):=\inf_{x>0}\{sx-H_p(x)\},
  \qquad
  L_p^{(\infty)}(s):=-H_p^*(s) = \sup_{x>0}\{-sx+H_p(x)\}.
\]
Then
\[
  \frac{d}{ds}H_p^*(s)=G_p^{-1}(s)=R_p^{(\infty)}(s)+\frac1s,
\]
where $R_p^{(\infty)}(s):=R_{\meas{p}}(s)$ denotes the Voiculescu $R$-transform of the empirical root distribution of $p$.

On the other hand, motivated by Proposition \ref{prop:Laplace-FFF}, define
\[
  L_p^{(N)}(s):=\frac1N\log \Lap{p}(Ns).
\]
To compare $L_p^{(N)}$ with $L_p^{(\infty)}$, fix $s\in(0,\alpha)$ and let $x_s>0$ be the unique point such that
\[
  G_p(x_s)=s.
\]
Set
\[
  \phi_s(x):=-sx+H_p(x),
  \qquad
  \psi_s(x):=\phi_s(x)-\phi_s(x_s).
\]
Since
\[
  \phi_s'(x)=-s+G_p(x),
  \qquad
  \phi_s''(x)=G_p'(x)<0,
\]
the function $\phi_s$ is strictly concave and attains its maximum at $x_s$.
Hence
\[
  L_p^{(\infty)}(s)=\phi_s(x_s)
\]
and $\psi_s(x) \le 0$.

\begin{proposition}\label{prop:Laplace-correction}
  For every $s\in(0,\alpha)$,
  \[
    L_p^{(N)}(s)
    =L_p^{(\infty)}(s)+\frac1N\log\int_0^\infty e^{N\psi_s(x)}\,dx.
  \]
\end{proposition}

\begin{proof}
  Since
  \[
    p(x)e^{-Nsx}=\exp\bigl(N\phi_s(x)\bigr),
  \]
  we have
  \[
  \begin{aligned}
    L_p^{(N)}(s)
    &=\frac1N\log\int_0^\infty e^{N\phi_s(x)}\,dx \\
    &=\frac1N\log\left(
      e^{N\phi_s(x_s)}
      \int_0^\infty e^{N(\phi_s(x)-\phi_s(x_s))}\,dx
    \right) \\
    &=\phi_s(x_s)+\frac1N\log\int_0^\infty e^{N\psi_s(x)}\,dx.
  \end{aligned}
  \]
\end{proof}

Thus the comparison between $L_p^{(N)}$ and $L_p^{(\infty)}$ reduces to estimating the integral
\[
  \int_0^\infty e^{N\psi_s(x)}\,dx.
\]
This observation reflects the Laplace principle underlying our approach.
Roughly speaking,
\[
  \frac{1}{N}\log\int_0^\infty e^{N\psi_s(x)}\,dx \longrightarrow 0
\]
as $N\to\infty$.
A related statement also appears in \cite[Eq.~(72)]{MergnyPotters2022}.
The main issue here is to obtain a quantitative bound.

\begin{definition}
  For $s\in(0,\alpha)$, define a probability measure on $[0,\infty)$ by
  \[
    \nu_{p,s}^{(N)}(dx)
    :=\frac{p(x)e^{-Nsx}\,dx}{\int_0^\infty p(x)e^{-Nsx}\,dx}
    =\frac{e^{N\psi_s(x)}\,dx}{\int_0^\infty e^{N\psi_s(x)}\,dx}.
  \]
\end{definition}

\begin{lemma}\label{lem:R-expectation}
  For $s\in(0,\alpha)$,
  \[
    R_p^{(N)}(s)
    =\frac1s\int_0^\infty xG_p(x)\,\nu_{p,s}^{(N)}(dx)-\frac1s.
  \]
  In particular,
  \[
    R_p^{(N)}(s)-R_p^{(\infty)}(s)
    =\frac1s\left(
      \int_0^\infty xG_p(x)\,\nu_{p,s}^{(N)}(dx)-sx_s
    \right).
  \]
\end{lemma}

\begin{proof}
  By Proposition \ref{prop:Laplace-FFF},
  \[
    R_p^{(N)}(s)
    =-\frac{d}{ds}L_p^{(N)}(s)-\left(1+\frac1N\right)\frac1s.
  \]
  Since $p(x)$ is a polynomial and $e^{-Nsx}$ decays exponentially for $s>0$, both
  \[
  \int_0^\infty p(x)e^{-Nsx}\,dx
  \quad\text{and}\quad
  \int_0^\infty x p(x)e^{-Nsx}\,dx
  \]
  are finite.
Hence we may differentiate under the integral and obtain
\[
  \frac{d}{ds}L_p^{(N)}(s)
  =-\frac{\int_0^\infty xp(x)e^{-Nsx}\,dx}
          {\int_0^\infty p(x)e^{-Nsx}\,dx}.
\]
  Thus,
  \begin{equation*}
        R_p^{(N)}(s)
    =\frac{\int_0^\infty xp(x)e^{-Nsx}\,dx}
            {\int_0^\infty p(x)e^{-Nsx}\,dx}
     -\left(1+\frac1N\right)\frac1s.
  \end{equation*}
  Integration by parts yields
  \[
    \int_0^\infty xp(x)e^{-Nsx}\,dx
    =\frac1{Ns}\left(
      \int_0^\infty p(x)e^{-Nsx}\,dx
      +\int_0^\infty xp'(x)e^{-Nsx}\,dx
    \right),
  \]
  so
  \begin{equation*}
        R_p^{(N)}(s)
    =\frac{\int_0^\infty xp'(x)e^{-Nsx}\,dx}
            {Ns\int_0^\infty p(x)e^{-Nsx}\,dx}
     -\frac1s.
  \end{equation*}
  Since $p'(x)/(Np(x))=G_p(x)$, this becomes
  \[
    R_p^{(N)}(s)
    = \frac{\int_0^\infty xG_p(x)p(x)e^{-Nsx}\,dx}
           {s\int_0^\infty p(x)e^{-Nsx}\,dx}
     -\frac1s,
  \]
  which is the first identity.
  The second follows from the definition:
  \[
    R_p^{(\infty)}(s)= G_p^{-1}(s) - \frac1s = x_s-\frac1s.
  \]
\end{proof}

Heuristically, the probability measure $\nu_{p,s}^{(N)}$ concentrates near $x_s$ as $N\to\infty$, and hence one expects $\int_0^\infty xG_p(x)\,\nu_{p,s}^{(N)}(dx)\sim x_sG_p(x_s)=sx_s$. 
In the next section, we justify this quantitatively by elementary estimates.



\begin{example} \label{ex:lateruse}
  We now consider the simplest case corresponding to the delta measure.
  Let
  \[
    p(x)=(x+\lambda)^N=\sum_{k=0}^N\binom{N}{k}\lambda^k x^{N-k},
    \qquad \lambda>0 .
  \]
  Then $\meas{p}=\delta_{-\lambda}$.
  In this case
  \[
    \widehat{p}(s)
    =\sum_{k=0}^N\frac{\lambda^k}{k!}s^k
    =\frac{s^{N+1}}{N!}\Lap{p}(s).
  \]

  Moreover,
  \[
    G_{\delta_{-\lambda}}(x)=\frac{1}{x+\lambda},
    \qquad
    H_{\delta_{-\lambda}}(x)=\log(x+\lambda),
  \]
  and therefore
  \[
    R_{\delta_{-\lambda}}(s)
    =G_{\delta_{-\lambda}}^{(-1)}(s)-\frac{1}{s}
    =-\lambda,
    \qquad 0<s<\frac{1}{\lambda}.
  \]

  Since $x_s=-\lambda+1/s$, we have
  \[
    L_{\delta_{-\lambda}}(s)
    =-H_{\delta_{-\lambda}}^*(s)
    =\phi_s(x_s)
    =s\lambda-\log s-1,
  \]
  where $\phi_s(x)=-sx+\log(x+\lambda)$.

  On the other hand,
  \[
    L_p^{(N)}(s)
    =\frac{1}{N}\log\int_0^\infty (x+\lambda)^N e^{-Nsx}\,dx .
  \]
  By the change of variables $y=x+\lambda$ and $z=sy$, we obtain
  \[
    \int_0^\infty (x+\lambda)^N e^{-Nsx}\,dx
    =e^{Ns\lambda}\int_\lambda^\infty y^N e^{-Nsy}\,dy
    =\frac{e^{Ns\lambda}}{s^{N+1}}\int_{s\lambda}^\infty z^N e^{-Nz}\,dz .
  \]

  Since $0<s\lambda<1$, putting $w=z-1$ yields
  \[
    \int_{s\lambda}^\infty z^N e^{-Nz}\,dz
    =\frac{1}{e^N}\int_{s\lambda-1}^\infty (w+1)^N e^{-Nw}\,dw .
  \]
  Hence
  \[
    L_p^{(N)}(s)-L_{\delta_{-\lambda}}(s)
    =\frac{1}{N}\log\!\left(
      \frac{1}{s}
      \int_{s\lambda-1}^\infty (w+1)^N e^{-Nw}\,dw
    \right).
  \]

  Since $-1<s\lambda-1<0$, we have
  \begin{equation*} 
        \int_{s\lambda-1}^\infty (w+1)^N e^{-Nw}\,dw
    \le
    \int_{-1}^\infty (w+1)^N e^{-Nw}\,dw
    =
    \frac{N!e^N}{N^{N+1}}
    \le
    e^{\frac{1}{12N}}\sqrt{\frac{2\pi}{N}},
  \end{equation*}
  where the last inequality follows from Stirling's formula.
  Moreover,
  \[
    \int_{s\lambda-1}^\infty (w+1)^N e^{-Nw}\,dw
    \ge
    \int_0^\infty (w+1)^N e^{-Nw}\,dw
    =
    \frac{N!e^N}{N^{N+1}}
    \cdot
    \frac{\sum_{k=0}^N N^k/k!}{e^N}.
  \]

  Consequently,
  \[
    L_p^{(N)}(s)-L_{\delta_{-\lambda}}(s)
    =
    -\frac{1}{N}\log s
    +\frac{1}{N}\!\left(
      \frac12\log(2\pi)-\frac12\log N
    \right)
    +o(N^{-1}),
  \]
  and also
  \[
    R_p^{(N)}(s)-R_{\delta_{-\lambda}}(s)
    =
    \frac{(s\lambda)^{N+1}e^{-N(s\lambda-1)}}
    {Ns\int_{s\lambda-1}^\infty (w+1)^N e^{-Nw}\,dw}.
  \]
\end{example}

\begin{remark}
  More generally, if $\lambda\in\R$ and $p(x)=(x-\lambda)^N$, then the definition gives
  \[
    R_p^{(N)}(s)
    = \lambda
      \frac{\sum_{k=0}^{N-1} \frac{(-1)^k}{k!}(\lambda Ns)^k}
           {\sum_{k=0}^N \frac{(-1)^k}{k!}(\lambda Ns)^k}.
  \]
  Although this formula follows directly from the definition, we include the above calculation because it provides a useful preview of the estimates that will appear later and because the resulting bound will be used again in the sequel.
\end{remark}

\section{\texorpdfstring{Convergence of the finite $R$-transform}{Convergence of the finite R-transform}}

In this section, we prove the main result.
After that, we state some of the applications.

\subsection{\texorpdfstring{Proof of Theorem~\ref{thm:main}}{Proof of the main theorem}}

\begin{lemma}\label{lem:psi-two-sided}
  Let
  \[
    \sigma_s^2:=-G_p'(x_s)=\frac1N\sum_{i=1}^N (x_s+\lambda_i)^{-2}.
  \]
  Then
  \[
    \sqrt{\frac{\pi}{2N\sigma_s^2}}
    \le
    \int_0^\infty e^{N\psi_s(x)}\,dx
    \le
    p(x_s)^{1/N}e^{1/(12N)}\sqrt{\frac{2\pi}{N}}.
  \]
\end{lemma}

\begin{proof}
We first prove the lower bound.
Since
\[
\begin{aligned}
  \psi_s(x)
  &= -s(x-x_s)+H_p(x)-H_p(x_s) \\
  &= \frac{1}{N}\sum_{i=1}^N
    \left(
      \log\left(1+\frac{x-x_s}{x_s+\lambda_i}\right)
      - \frac{x-x_s}{x_s+\lambda_i}
    \right),
\end{aligned}
\]
let $v(t):=\log(1+t)-t+\frac{t^2}{2}$ for $t>-1$,
and define $V(x-x_s)
  := \frac{1}{N}\sum_{i=1}^N
  v\!\left(\frac{x-x_s}{x_s+\lambda_i}\right)$ for $x>0$.
Then
\[
  \psi_s(x)
  = -\frac{1}{2}\sigma_s^2(x-x_s)^2 + V(x-x_s).
\]

Since $v(t)\ge 0$ for $t\ge 0$, we obtain $\psi_s(x)\ge -\frac{1}{2}\sigma_s^2(x-x_s)^2$ for $x \ge x_s$.
Hence
\[
\begin{aligned}
  \int_0^\infty e^{N\psi_s(x)}\,dx
  &\ge \int_{x_s}^\infty e^{N\psi_s(x)}\,dx \\
  &\ge \int_{x_s}^\infty e^{-N\sigma_s^2(x-x_s)^2/2}\,dx \\
  &= \int_0^\infty e^{-N\sigma_s^2x^2/2}\,dx \\
  &= \sqrt{\frac{\pi}{2N\sigma_s^2}}.
\end{aligned}
\]

For the upper bound, 
\[
\begin{aligned}
  e^{N\psi_s(x)}
  &= \exp\left(
       \sum_{i=1}^N
       \left(
         \log\left(1+\frac{x-x_s}{x_s+\lambda_i}\right)
         - \frac{x-x_s}{x_s+\lambda_i}
       \right)
     \right) \\
  &= \prod_{i=1}^N
     \left(1+\frac{x-x_s}{x_s+\lambda_i}\right)
     \exp\left(-\frac{x-x_s}{x_s+\lambda_i}\right) \\
  &= \frac{1}{p(x_s)}
     \prod_{i=1}^N
     (x+\lambda_i)\exp\left(-\frac{x-x_s}{x_s+\lambda_i}\right).
\end{aligned}
\]
Therefore, by Hölder's inequality,
\[
  \int_0^\infty e^{N\psi_s(x)}\,dx
  \le
  \frac{1}{p(x_s)}
  \prod_{i=1}^N
  \left(
    \int_0^\infty
    (x+\lambda_i)^N
    \exp\left(
      -N\frac{x-x_s}{x_s+\lambda_i}
    \right)\,dx
  \right)^{1/N}.
\]

  By the calculation in Example \ref{ex:lateruse}, for $\lambda>0$ and $0 < s < \lambda^{-1}$, we have 
  \[
  \int_0^\infty \frac{(x+\lambda)^N}{e^{Nsx}} dx = \frac{e^{Ns\lambda}}{s^{N+1}}  \int_{s\lambda}^{\infty} \frac{z^N}{e^{Nz}} dz \le \frac{e^{Ns\lambda}}{s^{N+1}}  \int_{0}^{\infty} \frac{z^N}{e^{Nz}} dz \le \frac{e^{N(s\lambda-1)}e^{\frac{1}{12N}}}{s^{N+1}}\sqrt{\frac{2\pi}{N}}.
  \]
  Applying this with $\lambda= \lambda_{i}$ and $s =(x_s + \lambda_i)^{-1}$ for each $i=1, \dots, N$, we have
  \[
  \int_0^\infty e^{N\psi_s(x)} dx \le p(x_s)^{\frac{1}{N}}e^{\frac{1}{12N}}\sqrt{\frac{2\pi}{N}}.
  \]
  This completes the proof.
\end{proof}

\begin{proposition}\label{prop:L}
Let $\epsilon > 0$. Assume that $p_N \in \calP_N((-\infty, -\epsilon])$ is a sequence of polynomials whose empirical root distributions $\meas{p_N}$ converge weakly to a probability measure $\mu \in \calM((-\infty, -\epsilon])$.
Then, for each $s\in(0,\alpha)$, where $\alpha:=G_\mu(0)$, we have
\[
  L_{p_N}^{(N)}(s)-L_{p_N}^{(\infty)}(s)
  =O\!\left(\frac{\log N}{N}\right).
\]
\end{proposition}

\begin{proof}
  This follows immediately from Proposition \ref{prop:Laplace-correction}
  and Lemma \ref{lem:psi-two-sided}.
\end{proof}

We include Proposition~\ref{prop:L} to make explicit the sense in which $L_p^{(N)}$ approximates $L_p^{(\infty)}$. Although it is not used directly in the proof of the main theorem, it may be of independent interest.

\begin{lemma}\label{lem:comparison-kernel}
  \begin{enumerate}
    \item   Define
    \[
    T_p(x):=\frac1N\sum_{i=1}^N
    \frac{\lambda_i}{(x+\lambda_i)(x_s+\lambda_i)}.
  \]
  Then $T_p$ is monotonically decreasing, $0\le T_p(x)\le s$, and
  \[
    xG_p(x)-x_sG_p(x_s)=T_p(x)(x-x_s).
  \]
    \item   For $x\ge 0$, we have
  \[
    -\psi_s'(x)\le \sigma_s^2(x-x_s)\le -\frac{x}{x_s}\psi_s'(x).
  \]
  \end{enumerate}
\end{lemma}

\begin{proof}
  We have
  \[
  \begin{aligned}
    xG_p(x)-x_sG_p(x_s)
    &= \frac1N\sum_{i=1}^N
       \left(
         \frac{x}{x+\lambda_i}-\frac{x_s}{x_s+\lambda_i}
       \right) \\
    &= \frac1N\sum_{i=1}^N
       \frac{\lambda_i(x-x_s)}{(x+\lambda_i)(x_s+\lambda_i)} \\
    &= T_p(x)(x-x_s),
  \end{aligned}
  \]
  which gives the first claim.
  Since each summand is nonnegative and
  \[
    \frac{\lambda_i}{(x+\lambda_i)(x_s+\lambda_i)}
    \le \frac{1}{x_s+\lambda_i},
  \]
  we obtain
  \[
    0\le T_p(x)\le \frac1N\sum_{i=1}^N \frac{1}{x_s+\lambda_i}=G_p(x_s)=s.
  \]

  Next,
  \[
    \psi_s'(x)= -s+G_p(x)=G_p(x)-G_p(x_s),
  \]
  and
  \[
  \begin{aligned}
    G_p(x)-G_p(x_s)
    &= -\frac1N\sum_{i=1}^N
       \frac{x-x_s}{(x+\lambda_i)(x_s+\lambda_i)} \\
    &= -\sigma_s^2(x-x_s)
       + \frac1N\sum_{i=1}^N
         \frac{(x-x_s)^2}{(x+\lambda_i)(x_s+\lambda_i)^2}.
  \end{aligned}
  \]
  Since the correction term is nonnegative, we obtain
  \[
    -\psi_s'(x)\le \sigma_s^2(x-x_s).
  \]
  Moreover,
  \[
    \frac{(x-x_s)^2}{(x+\lambda_i)(x_s+\lambda_i)^2}
    \le
    \frac{x-x_s}{x_s}\cdot
    \frac{x-x_s}{(x+\lambda_i)(x_s+\lambda_i)},
  \]
  and summing over $i$ yields
  \[
    \sigma_s^2(x-x_s)\le -\frac{x}{x_s}\psi_s'(x).
  \]
\end{proof}

\begin{proposition}\label{prop:R-comparison}
  For $s\in(0,\alpha)$, we have
\[
-\frac{\alpha}{sN x_s \sigma_s^2}
    \le
    R_p^{(N)}(s)-R_p^{(\infty)}(s)
    \le
    \frac{1}{N x_s \sigma_s^2}.
\]
\end{proposition}

\begin{proof}
  By Lemma \ref{lem:R-expectation}, it suffices to estimate
  \[
    A:=\int_0^\infty \bigl(xG_p(x)-x_sG_p(x_s)\bigr)e^{N\psi_s(x)}\,dx.
  \]
  By Lemma \ref{lem:comparison-kernel},
  \[
    A=\int_0^\infty T_p(x)(x-x_s)e^{N\psi_s(x)}\,dx.
  \]
  Using
  \[
    -\psi_s'(x)\le \sigma_s^2(x-x_s)\le -\frac{x}{x_s}\psi_s'(x),
  \]
  we obtain
  \[
    \int_0^\infty -T_p(x)\psi_s'(x)e^{N\psi_s(x)}\,dx
    \le
    \sigma_s^2 A
    \le
    \frac1{x_s}\int_0^\infty -xT_p(x)\psi_s'(x)e^{N\psi_s(x)}\,dx.
  \]

  Multiplying by $N$ and applying integration by parts to the left-hand side, we obtain
  \[
  \begin{aligned}
    N\sigma_s^2 A
    &\ge -\bigl[T_p(x)e^{N\psi_s(x)}\bigr]_{0}^{\infty}
         + \int_0^\infty T_p'(x)e^{N\psi_s(x)}\,dx \\
    &= T_p(0)e^{N\psi_s(0)}
       + \int_0^\infty T_p'(x)e^{N\psi_s(x)}\,dx.
  \end{aligned}
  \]
  Since $T_p(0)=s$ and
  \[
    -T_p'(x)
    = \frac1N\sum_{i=1}^N
      \frac{\lambda_i}{(x+\lambda_i)^2(x_s+\lambda_i)}
    \le \frac{1}{x_s}\frac1N\sum_{i=1}^N \frac{1}{\lambda_i}
    \le \frac{\alpha}{x_s},
  \]
  we get
  \[
    N\sigma_s^2 A
    \ge
    -\frac{\alpha}{x_s}\int_0^\infty e^{N\psi_s(x)}\,dx.
  \]

  Similarly, on the right-hand side,
  \[
  \begin{aligned}
    x_sN\sigma_s^2 A
    &\le -\bigl[xT_p(x)e^{N\psi_s(x)}\bigr]_0^\infty
         + \int_0^\infty (T_p(x)+xT_p'(x))e^{N\psi_s(x)}\,dx \\
    &\le \int_0^\infty T_p(x)e^{N\psi_s(x)}\,dx \\
    &\le s\int_0^\infty e^{N\psi_s(x)}\,dx.
  \end{aligned}
  \]

  Therefore
  \[
    -\frac{\alpha}{sN x_s \sigma_s^2}
    \le
    \frac{A}{s\int_0^\infty e^{N\psi_s(x)}\,dx}
    \le
    \frac{1}{N x_s \sigma_s^2}.
  \]
  By Lemma \ref{lem:R-expectation}, this is exactly
  \begin{equation*} 
    -\frac{\alpha}{sN x_s \sigma_s^2}
    \le
    R_p^{(N)}(s)-R_p^{(\infty)}(s)
    \le
    \frac{1}{N x_s \sigma_s^2}.
  \end{equation*}
  This is what we wanted.
\end{proof}

\begin{proof}[Proof of Theorem~\ref{thm:main}]
  Put $\alpha_N:=G_{\meas{p_N}}(0)$.
  Since $\meas{p_N}\weakto\mu$, we have $\alpha_N\to\alpha$.
  Hence, for each fixed $s\in(0,\alpha)$, we have $s\in(0,\alpha_N)$ for all sufficiently large $N$.

  Applying Proposition~\ref{prop:R-comparison} to $p_N$, and using again the weak convergence $\meas{p_N}\weakto\mu$, the quantities appearing in the bound converge to their counterparts for $\mu$.
  Thus
  \[
    R_{p_N}^{(N)}(s)-R_{\meas{p_N}}(s)=O(N^{-1}),
  \]
  as desired.
\end{proof}

\subsection{Applications}

We conclude the paper with several applications of Theorem~\ref{thm:main}.

\begin{corollary} \label{cor:convergenceOfFiniteFreeConvolution}
  Let $\epsilon > 0$, and let $(p_N)_{N=1}^\infty$ be a sequence of polynomials with $p_N \in \calP_N((-\infty, -\epsilon])$.
   \begin{enumerate}
    \item If the empirical root distributions $\meas{p_N}$ converge weakly to a probability measure $\mu$, then $R_{p_N}^{(N)}(s)$ converges to $R_\mu(s)$ for $s \in (0, G_\mu(0))$.
    \item  Suppose, in addition, that there exists $M>\epsilon$ such that $p_N\in\calP_N([-M,-\epsilon])$ for all $N$.
    If the finite $R$-transforms $R_{p_N}^{(N)}(s)$ converge to a function $R(s)$ defined on $(0,1/M)$, then there exists a probability measure $\mu \in \calM([-M, -\epsilon])$ such that $\meas{p_N} \weakto \mu$ and $R_\mu(s) = R(s)$.
   \end{enumerate}
\end{corollary}
\begin{proof}
  These assertions follow from Theorem~\ref{thm:main}, together with the standard fact that, for compactly supported measures, weak convergence is equivalent to convergence of the Voiculescu $R$-transforms near the origin.
\end{proof}

\begin{corollary} \label{cor:finitefreeadditive}
  Let $0<a<b$, and let $p_N,q_N \in \calP_N([-b,-a])$ satisfy
  \[
    \meas{p_N}\weakto\mu,
    \qquad
    \meas{q_N}\weakto\nu
  \]
  for probability measures $\mu,\nu \in \calM(\R_{\le 0})$.
  Then
  \[
    \meas{p_N \boxplus_N q_N} \weakto \mu \boxplus \nu.
  \]
\end{corollary}

For the proof of Corollary~\ref{cor:finitefreeadditive}, we use the following inequality.

\begin{proposition} \label{prop:inequalityR}
  Let $N\in\N$, and let $p,q \in \calP_N(\R_{<0})$.
  Then
  \[
    R^{(N)}_{p \boxplus_N q}(s)
    \ge
    R^{(N)}_{p}(s) + R^{(N)}_{q}(s)
  \]
  for all $s>0$.
\end{proposition}

This should be contrasted with \cite[Theorem~1.12]{MSS22}, which gives an inequality in the opposite direction:
\begin{equation} \label{eq:inequality}
    R_{p \boxplus_N q}^{(\infty)}(s)
    \le
    R_p^{(\infty)}(s) + R_q^{(\infty)}(s).
\end{equation}

\begin{proof}
  Put $r = p \boxplus_N q$ and $f = \widehat{p} \widehat{q} - \widehat{r}.$
  Since $\widehat{r}(s)\overset{N}{=}\widehat{p}(s)\widehat{q}(s)$, the polynomial $f$ has the form
  \[
    f(s) = \sum_{k=N+1}^{2N} \frac{d_k}{k!}s^k,
  \]
  where the coefficients $d_k =
    (-1)^k
    \sum_{i+j=k}
    \binom{k}{i}
    \coef_i(p)\coef_j(q)$ are positive.
  We also write $\widehat{r}$ in a similar form:
  \[
    \widehat{r}(s) = \sum_{k=0}^{N} \frac{c_k}{k!}s^k,
  \]
  with $c_k>0$.

  Taking logarithms, we obtain
  \[
    \log \widehat{p}(s) + \log \widehat{q}(s)
    =
    \log \widehat{r}(s) - \log (1 - g(s)),
  \]
  where $g(s) := {f(s)}/{\widehat{p}(s)\widehat{q}(s)}.$
  Note that $0< 1 - g(s) = {\widehat{r}(s)}/{\widehat{p}(s)\widehat{q}(s)} < 1$ for $s>0$.
  Differentiating, we obtain
  \[
    \frac{\widehat{p}'(s)}{\widehat{p}(s)}
    +
    \frac{\widehat{q}'(s)}{\widehat{q}(s)}
    =
    \frac{\widehat{r}'(s)}{\widehat{r}(s)}
    +
    \frac{g'(s)}{1-g(s)}.
  \]
  By the definition of the finite $R$-transform, this yields
  \begin{equation} \label{eq:equality}
        R_{p}^{(N)}(s) + R_{q}^{(N)}(s)
    =
    R_{r}^{(N)}(s)
    -
    \frac{g'(Ns)}{1-g(Ns)}.
  \end{equation}
  Thus it remains to show that $g'(s)>0$ for $s>0$.
Indeed,
\[
  g'(s)
  =
  \frac{
    f'(s)(\widehat{p}\widehat{q})(s)
    -
    f(s)(\widehat{p}\widehat{q})'(s)
  }
  {(\widehat{p}(s)\widehat{q}(s))^2}
  =
  \frac{
    f'(s)\widehat{r}(s)
    -
    f(s)\widehat{r}'(s)
  }
  {(\widehat{p}(s)\widehat{q}(s))^2}.
\]
  Moreover,
\begin{align}
  f'(s) \widehat{r}(s) - f(s) \widehat{r}'(s)
  &=
  \left(
    \sum_{k=N+1}^{2N} \frac{d_k}{k!} k s^{k-1}
  \right)
  \left(
    \sum_{l=0}^{N} \frac{c_l}{l!} s^l
  \right) \notag \\
  &\quad -
  \left(
    \sum_{k=N+1}^{2N} \frac{d_k}{k!} s^k
  \right)
  \left(
    \sum_{l=0}^{N} \frac{c_l}{l!} l s^{l-1}
  \right) \notag \\
  &=
  \sum_{k=N+1}^{2N}\sum_{l=0}^{N}
  \frac{d_k c_l}{k!l!}
  (k-l)s^{k+l-1}. \label{eq:expression}
\end{align}
Since $k-l>0$ in the last sum, Eq.~\eqref{eq:expression} implies $g'(s)>0$ for $s>0$.
Hence the correction term $g'(Ns)/(1-g(Ns))$ in Eq. \eqref{eq:equality} is positive, and the desired inequality follows.
\end{proof}

\begin{proof}[Proof of Corollary~\ref{cor:finitefreeadditive}]
  Set $r_N:=p_N\boxplus_N q_N$. By Walsh's theorem, or by \cite[Theorem~1.3]{MSS22}, we have
  \[
    r_N\in \calP_N([-2b,-2a]).
  \]
  Hence the sequence $(\meas{r_N})$ is tight, and every subsequence has a weakly convergent subsequence.

  Let $(\meas{r_{N_k}})$ be such a subsequence and let $\meas{r_{N_{k_l}}}\weakto \rho$. By Corollary~\ref{cor:convergenceOfFiniteFreeConvolution} and Proposition~\ref{prop:inequalityR}, we obtain
  \[
    R_\rho(s)
    \ge
    R_\mu(s)+R_\nu(s)
  \]
  for $s \in (0, 1/2b)$. On the other hand, \cite[Theorem~1.12]{MSS22}, in the form \eqref{eq:inequality}, gives the reverse inequality
  \[
    R_\rho(s)
    \le
    R_\mu(s)+R_\nu(s).
  \]
  Therefore
  \[
    R_\rho(s)=R_\mu(s)+R_\nu(s)=R_{\mu\boxplus\nu}(s)
  \]
  for $s \in (0, 1/2b)$, and hence $\rho=\mu\boxplus\nu$.
  Since every subsequence of $(\meas{r_N})$ has a further subsequence converging to $\mu\boxplus\nu$, we conclude that
  \[
    \meas{p_N\boxplus_N q_N}\weakto \mu\boxplus\nu.
  \]
\end{proof}

\section*{Acknowledgements}

The authors gratefully acknowledge the financial support from the grant CONAHCYT A1-S-9764 and JSPS Open Partnership Joint Research Projects Grant (No.~JPJSBP120209921).
K.F. is supported by JSPS Research Fellowship for Young Scientists PD (KAKENHI Grant No.~24KJ1318).
This work was primarily carried out during the second author's stay at CIMAT as a JSPS postdoctoral fellow affiliated with Kyoto University.

\vspace{6mm}

\begin{enumerate}
  \item[]
  \hspace{-10mm} Octavio Arizmendi\\
  Centro de Investigación en Matemáticas, Guanajuato, Gto. 36000, Mexico\\
  Email: octavius@cimat.mx

  \item[]
  \hspace{-10mm} Katsunori Fujie\\
  Department of Information Science and Technology, Aichi Prefectural University\\
  Nagakute, Aichi, 480-1198, Japan\\
  Email: kfujie@ist.aichi-pu.ac.jp
\end{enumerate}

\end{document}